\documentclass[oneside,english,final]{amsart}
\usepackage[T1]{fontenc}
\usepackage[latin9]{inputenc}
\usepackage{color}
\usepackage{babel}
\usepackage{refstyle}
\usepackage{amsbsy}
\usepackage{amstext}
\usepackage{amsthm}
\usepackage{amssymb}
\usepackage[unicode=true,pdfusetitle,
 bookmarks=true,bookmarksnumbered=false,bookmarksopen=false,
 breaklinks=false,pdfborder={0 0 1},backref=false,colorlinks=true]
 {hyperref}
\hypersetup{
 linkcolor=red, urlcolor=red, citecolor=red}

\makeatletter

\AtBeginDocument{\providecommand\thmref[1]{\ref{thm:#1}}}
\AtBeginDocument{\providecommand\propref[1]{\ref{prop:#1}}}
\AtBeginDocument{\providecommand\lemref[1]{\ref{lem:#1}}}
\RS@ifundefined{subsecref}
  {\newref{subsec}{name = \RSsectxt}}
  {}
\RS@ifundefined{thmref}
  {\def\RSthmtxt{theorem~}\newref{thm}{name = \RSthmtxt}}
  {}
\RS@ifundefined{lemref}
  {\def\RSlemtxt{lemma~}\newref{lem}{name = \RSlemtxt}}
  {}

\numberwithin{equation}{section}
\numberwithin{figure}{section}
\newenvironment{lyxlist}[1]
	{\begin{list}{}
		{\settowidth{\labelwidth}{#1}
		 \setlength{\leftmargin}{\labelwidth}
		 \addtolength{\leftmargin}{\labelsep}
		 }}
	{\end{list}}
\theoremstyle{plain}
\newtheorem{thm}{\protect\theoremname}
\theoremstyle{remark}
\newtheorem{rem}{\protect\remarkname}
\theoremstyle{plain}
\newtheorem{prop}{\protect\propositionname}
\theoremstyle{plain}
\newtheorem*{thm*}{\protect\theoremname}
\theoremstyle{plain}
\newtheorem{lem}{\protect\lemmaname}

\makeatother

\providecommand{\lemmaname}{Lemma}
\providecommand{\propositionname}{Proposition}
\providecommand{\remarkname}{Remark}
\providecommand{\theoremname}{Theorem}

\begin{document}

\title{Scattering for NLS with a sum of two repulsive potentials}

\author{David Lafontaine}
\thanks{d.lafontaine@bath.ac.uk, University of Bath, Department of Mathematical Sciences}

\begin{abstract}
We prove the scattering for a defocusing nonlinear Schrödinger equation
with a sum of two repulsive potentials with strictly convex level
surfaces, thus providing a scattering result in a trapped setting
similar to the exterior of two strictly convex obstacles.
\end{abstract}

\maketitle

\section{Introduction}

We are concerned by the following defocusing non-linear Schrödinger
equation with a potential

\begin{equation}
i\partial_{t}u+\Delta u-Vu=u|u|^{\alpha},\;u(0)=\varphi\in H^{1}.\label{eq:nlsv}
\end{equation}
in arbitray spatial dimension $d\geq1$. Once good dispersive properties
of the linear flow, such as Strichartz estimates described below in
the paper, are established, the local well-posedness of (\ref{eq:nlsv})
follows by usual fixed point arguments. Because of the energy conservation
law,
\[
E(u(t)):=\frac{1}{2}\int|\nabla u(t)|^{2}+\int V|u(t)|^{2}+\frac{1}{\alpha+2}\int|u(t)|^{\alpha+2}=E(u(0))
\]
this result extends to global well-posedness. Thus, it is natural
to investigate the asymptotic behavior of solutions of (\ref{eq:nlsv}).

It is well-known since Nakanishi's paper \cite{MR1726753} that for
$V=0$, in the intercritical regime
\begin{equation}
\frac{4}{d}<\alpha<\begin{cases}
+\infty & d=1,2,\\
\frac{4}{d-2} & d\geq3,
\end{cases}\label{eq:intercrit}
\end{equation}
 the solutions \textit{scatter} in $H^{1}(\mathbb{R}^{d})$, that
is, for every solution $u\in C(\mbox{\ensuremath{\mathbb{R}},}H^{1}(\mathbb{R}^{d}))$
of (\ref{eq:nlsv}), there exists a unique couple of data $\psi_{\pm}\in H^{1}(\mathbb{R}^{d})$
such that
\[
\Vert u(t)-e^{-it\Delta}\psi_{\pm}\Vert_{H^{1}(\mathbb{R}^{d})}\underset{t\rightarrow\pm\infty}{\longrightarrow0}.
\]

The inhomogeneous setting $V\neq0$ was investigated more recently,
for example in \cite{MR213495011}, \cite{AsymptAn}, \cite{MR2134950},
\cite{Forcella}. However, all these scattering results rely on a
non-trapping assumption, namely, that the potential is \emph{repulsive:
\[
x\cdot\nabla V\leq0,
\]
}or, as in \cite{Carles}, that its non-repulsive part is sufficently
small. The aim of this paper is to establish a scattering result in
a trapping situation. More precisely, we are interested in one of
the simplest unstable trapping framework, that is, the case where
$V$ is the sum of two positive, repulsive potentials with strictly
convex level surfaces. It is the potential-analog of the homogeneous
problem outside two strictly convex obstacles, and this note can be
seen as a proxy for the scattering outside two strictly convex obstacles.
This more intricate problem, where reflexions at the boundary have
to be dealt with, will be treated in \cite{LafLaurent} using ideas
developped here. Note that the case of the exterior of a star-shaped
obstacle was treated by \cite{PlanchonVega09}, \cite{IvaPlanchon10},
\cite{PlanchonVega12}.

Let us precise our setting. Let $V_{1}$ and $V_{2}$ be two positive,
smooth potentials. We will denote by $V=V_{1}+V_{2}$ the total potential.
We make the following geometrical assumptions:
\begin{lyxlist}{00.00.0000}
\item [{(G1)\label{(G1)}}] $V_{1}$ and $V_{2}$ are repulsive, that is,
there exists $a_{1}$ and $a_{2}$ in $\mathbb{R}^{d}$ such that
\[
(x-a_{1,2})\cdot\nabla V_{1,2}\leq0.
\]
Without loss of generality, we assume that $0\in[a_{1},a_{2}]$.
\item [{(G2)\label{(G2)}}] The level surfaces of $V_{1}$ and $V_{2}$
are convex, and uniformly strictly convex in the non-repulsive region:
the eigenvalues of their second fundamental forms are uniformly bounded
below by a strictly positive universal constant in $\left\{ x\cdot\nabla V>0\right\} $.
\item [{(G3)\label{(G3)}}] All the trapped trajectories of the Hamiltonian
flow associated with $-\Delta+V$ belong to a same line $\mathcal{R}\subset\left\{ x\cdot\nabla V>0\right\} $
: for any pair $\Theta_{1},\Theta_{2}$ of level surfaces of $V_{1}$
and $V_{2}$, the unique trapped ray of the geometrical optics of
$\mathbb{R}^{d}\backslash\Theta_{1}\cup\Theta_{2}$ is included in
$\mathcal{R}$.
\end{lyxlist}
A non-trivial example of potential $V=V_{1}+V_{2}$ verifying (G1)-(G2)-(G3)
is given by $V_{1}(x):=e^{-|x-a|^{2}}$, $V_{2}(x):=e^{-|x+a|^{2}}$,
the uniformity in the convexity assumption coming from the fact that
this potential has a bounded non-repulsive region.

We will, in addition, assume the following decay assumption
\begin{gather}
V\in L^{\frac{d}{2}}((1+|x|^{\beta})dx);\ \nabla V_{1},\nabla V_{2}\in L^{\frac{d}{2}};\ \nabla V\in L^{\frac{d}{2}}(|x|^{\beta}dx),\label{eq:decay_assump}\\
\text{with }\beta>\frac{4}{3}.\nonumber 
\end{gather}
 It is the (improved) multi-dimensional analog of the decay assumption
arising in \cite{AsymptAn}. And finally, that the pointwise dispersive
estimate
\begin{equation}
\Vert e^{it(-\Delta+V)}\Vert_{L^{1}\longrightarrow L^{\infty}}\lesssim\frac{1}{|t|^{d/2}}\label{eq:disp}
\end{equation}
holds. Note that, in the same way as remarked in \cite{AsymptAn}
for the one dimensional case, this last assumption is automaticaly
verified using Goldberg and Schlag's result \cite{MR2096737} under
the non-negativity and decay assumptions with $\beta\geq2$ in dimension
$d=3$. Our main result reads
\begin{thm}
\label{thm:main}Assume that $d\geq3$. Let $V_{1}$ and $V_{2}$
be two positive, repulsive (G1) smooth potentials, with convex and
uniformly strictly convex in the non-repulsive region level surfaces
(G2), and colinear trapped trajectories (G3). Assume moreover that
$V=V_{1}+V_{2}$ verifies the decay assumption (\ref{eq:decay_assump}),
and the dispersive estimate (\ref{eq:disp}). Then, in the intercritical
regime (\ref{eq:intercrit}), every solutions of (\ref{eq:nlsv})
with potential $V=V_{1}+V_{2}$ scatter in $H^{1}(\mathbb{R}^{d})$.
\end{thm}
As in the aforementioned papers, we use the strategy of concentration,
compactness and rigidity first introduced by Kenig and Merle in \cite{MR2257393}:
assuming that there exists a finite energy above which solutions do
not scatter, one constructs a compact-flow solution and eliminates
it. Notice that in the case of a repulsive potential, this last rigidity
part is immediate by classical Morawetz estimates. It will be here
the main difficulty to overcome and the novelty of this note. After
some preliminaries, we construct a critical solution in the second
section, following \cite{AsymptAn} and generalizing it to any spatial
dimension. In the last section, we eliminate it using a family of
Morawetz mulipliers for which the gradient almost vanishes on the trapped
trajectory.

\begin{rem}
We assume that $d\neq2$ because our proof relies on endpoint Strichartz
estimates that are not true in dimension two, and the convexity assumption
we make on the potentials have no sense in the one dimensional case.
\end{rem}
\begin{rem}
The first two sections of this paper generalize in particular the
one-dimensional result of \cite{AsymptAn}, to any spatial dimension
$d\geq3$.
\end{rem}
\begin{rem}
As mentionned earlier, the geometrical framework (G1)-(G2)-(G3) is
in many aspects the potential-analog of the homogeneous problem outside
two strictly convex obstacles. This is the subject of a work in progress
\cite{LafLaurent}. A rigidity argument in the particular case of
two balls for the energy critical wave equation can be found in \cite{Waves}. 
\end{rem}
\begin{rem}
It is straightforward from the last section that the result is still
valid for an arbitrary finite sum of convex repulsive potentials $V=V_{1}+\cdots+V_{N}$
for which all trapped trajectories are included in the same line.
However, we present the proof for only two potentials in the seek
of simplicity.
\end{rem}

\section{Preliminaries}

\subsection{Usefull exponents}

From now on, we will fix the three following Strichartz exponents
\[
r=\alpha+2,\;q=\frac{2\alpha(\alpha+2)}{d\alpha^{2}-(d-2)\alpha-4},\;p=\frac{2(\alpha+2)}{4-(d-2)\alpha}.
\]
Moreover, let $\eta$ be the conjugate of the critical exponent $2^{\star}$:
\begin{equation}
\frac{1}{2^{\star}}+\frac{1}{\eta}=1.\label{eq:eta0}
\end{equation}
Notice, for the sequel, the following two identities
\begin{equation}
\frac{2}{d}+\frac{1}{2^{\star}}=\frac{1}{\eta},\label{eq:eta}
\end{equation}
and
\begin{equation}
\frac{2}{d}+\frac{2}{2^{\star}}=1.\label{eq:delta}
\end{equation}
Finally, let $\gamma$ be such that $(\gamma,\eta')$ follows the
admissibility condition of Theorem 1.4 of Foshi's inhomogeneous Strichartz
estimates \cite{MMR2134950}. Note that, in the intercritical regime
(\ref{eq:intercrit}), all these exponents are well defined and larger
than one.

\subsection{Strichartz estimates}

Let us recall that $e^{-it(-\Delta+V)}$ verifies the pointwise dispersive
estimates (\ref{eq:disp}), by \cite{MR2096737} in dimension $d=3$
for $\beta\geq2$, or by assumption in other cases. Interpolating
it with the mass conservation law, we obtain immediatly for all $a\in[2,\infty]$
\begin{equation}
\Vert e^{it(-\Delta+V)}\psi\Vert_{L^{a}}\lesssim\frac{1}{|t|^{\frac{d}{2}(\frac{1}{a'}-\frac{1}{a})}}\Vert\psi\Vert_{L^{a'}}.\label{eq:dispest}
\end{equation}
Moreover, it leads by the classical $TT^{\star}$ method (see for
example \cite{MR1646048}) to the Strichartz estimates
\begin{equation}
\Vert e^{-it(-\Delta+V)}\varphi\Vert_{L^{q_{1}}L^{r_{1}}}+\Vert\int_{0}^{t}e^{-i(t-s)(-\Delta+V)}F(s)ds\Vert_{L^{q_{2}}L^{r_{2}}}\lesssim\Vert\varphi\Vert_{L^{2}}+\Vert F\Vert_{L^{q_{3}'}L^{r_{3}'}}\label{eq:adst}
\end{equation}
for all pairs $(q_{i},r_{i})$ satisfying the admissibility condition
\[
\frac{2}{q_{i}}+\frac{d}{r_{i}}=\frac{d}{2},\ (q_{i},r_{i},d)\neq(2,\infty,2).
\]
We will use moreover the following Strichartz estimates associated
to non admissible pairs:
\begin{prop}[Strichartz estimates]
\label{strichartz}For all $\varphi\in H^{1}$, all $F\in L^{q'}L^{r'}$,
all $G\in L^{q'}L^{r'}$ and all $H\in L^{\gamma'}L^{\eta}$
\begin{equation}
\Vert e^{-it(-\Delta+V)}\varphi\Vert_{L^{p}L^{r}}\lesssim\Vert\varphi\Vert_{H^{1}}\label{eq:st2}
\end{equation}
\begin{equation}
\Vert\int_{0}^{t}e^{-i(t-s)(-\Delta+V)}F(s)ds\Vert_{L^{\alpha}L^{\infty}}\lesssim\Vert F\Vert_{L^{q'}L^{r'}}\label{eq:st3}
\end{equation}

\begin{equation}
\Vert\int_{0}^{t}e^{-i(t-s)(-\Delta+V)}G(s)ds\Vert_{L^{p}L^{r}}\lesssim\Vert G\Vert_{L^{q'}L^{r'}}\label{eq:st4}
\end{equation}

\begin{equation}
\Vert\int_{0}^{t}e^{-i(t-s)(-\Delta+V)}H(s)ds\Vert_{L^{p}L^{r}}\lesssim\Vert H\Vert_{L^{\gamma'}L^{\eta}}.\label{eq:st4-1}
\end{equation}
\end{prop}
\begin{proof}
The estimate (\ref{eq:st2}) follows from admissible Strichartz estimate
\[
\Vert e^{-it(-\Delta+V)}\varphi\Vert_{L^{p}L^{\frac{2dp}{dp-4}}}\lesssim\Vert\varphi\Vert_{L^{2}}
\]
together with a Sobolev embedding. The estimate (\ref{eq:st4}) is
contained in Lemma 2.1 of \cite{MR1171761}. Finally, (\ref{eq:st3})
and (\ref{eq:st4-1}) enters on the frame of the non-admissible inhomogheneous
Strichartz estimates of Theorem 1.4 of Foschi's paper \cite{MMR2134950}.
\end{proof}

\subsection{Perturbative results}

The three following classical perturbative results, follow immediatly
from the previous Strichartz inequalities with exact same proof as
in \cite{AsymptAn}.
\begin{prop}
\label{pert1}Let $u\in C(H^{1})$ be a solution of (\ref{eq:nlsv}).
If $u\in L^{p}L^{r}$, then $u$ scatters in $H^{1}$. 
\end{prop}

\begin{prop}
\label{pert2}There exists $\epsilon_{0}>0$, such that, for every
data $\varphi\in H^{1}$ such that $\Vert\varphi\Vert_{H^{1}}\leq\epsilon_{0}$,
the corresponding maximal solution of (\ref{eq:nlsv}) scatters in
$H^{1}$.
\end{prop}
\begin{prop}
\label{pert3}For every $M>0$ there exists $\epsilon>0$ and $C>0$
such that the following occurs. Let $v\in C(H^{1})\cap L^{p}L^{r}$
be a solution of the following integral equation with source term
$e(t,x)$ 
\[
v(t)=e^{-it(\Delta-V)}\varphi-i\int_{0}^{t}e^{-i(t-s)(\Delta-V)}(v(s)|v(s)|^{\alpha})ds+e(t)
\]
 with $\Vert v\Vert_{L^{p}L^{r}}<M$ and $\Vert e\Vert_{L^{p}L^{r}}<\epsilon$.
Assume moreover that $\varphi_{0}\in H^{1}$ is such that $\Vert e^{-it(\Delta-V)}\varphi_{0}\Vert_{L^{p}L^{r}}<\epsilon$
. Then, the solution $u\in C(H^{1})$ to (\ref{eq:nlsv}) with initial
condition $\varphi+\varphi_{0}$ satisfies
\[
u\in L^{p}L^{r},\quad\Vert u-v\Vert_{L^{p}L^{r}}<C.
\]
 
\end{prop}

\section{Construction of a critical solution}

The aim of this section is to extend the construction of a critical
element of \cite{AsymptAn} to any dimension $d\neq2$ -- no repulsivity
assumption is used in this first part of this work. This previous
paper follows itself \cite{MR213495011} which deals with a Dirac
potential, which is more singular but for which explicit formulas
are at hand. More precisely, let 
\begin{multline}
E_{c}=\sup\big\{ E>0\ |\ \forall\varphi\in H^{1},\ E(\varphi)<E\Rightarrow \\ \text{the solution of \ensuremath{(1.1)} with data }\varphi\text{ is in }L^{p}L^{r}\big\} .
\end{multline}
We will prove
\begin{thm}
\label{thm:critsol}If $E_{c}<\infty$, then there exists $\varphi_{c}\in H^{1}$,
$\varphi_{c}\neq0$, such that the corresponding solution $u_{c}$
of (\ref{eq:nlsv}) has a relatively compact flow $\left\{ u_{c}(t),\ t\geq0\right\} $
in $H^{1}$ and does not scatter.
\end{thm}
We assume all along this section that $d\geq3$.

\subsection{Profile decomposition}

We first show, with the same method as in \cite{AsymptAn}, extended
to any dimension, that we can use the abstract profile decomposition
obtained by \cite{MR213495011}:
\begin{thm*}[Astract profile decomposition, \cite{MR213495011}]
 Let $A:L^{2}\supset D(A)\rightarrow L^{2}$ be a self adjoint operator
such that:

\begin{itemize}
\item for some positive constants $c,C$ and for all $u\in D(A)$, 
\begin{equation}
c\Vert u\Vert_{H^{1}}^{2}\leq(Au,u)+\Vert u\Vert_{L^{2}}^{2}\leq C\Vert u\Vert_{H^{1}}^{2},\label{eq:ab1}
\end{equation}
\item let $B:D(A)\times D(A)\ni(u,v)\rightarrow(Au,v)+(u,v)_{L^{2}}-(u,v)_{H^{1}}\in\mathbb{C}$.
Then, as $n$ goes to infinity
\begin{equation}
B(\tau_{x_{n}}\psi,\tau_{x_{n}}h_{n})\rightarrow0\quad\forall\psi\in H^{1}\label{eq:ab2}
\end{equation}
as soon as
\[
x_{n}\rightarrow\pm\infty,\quad\sup\Vert h_{n}\Vert_{H^{1}}<\infty
\]
or
\[
x_{n}\rightarrow\bar{x}\in\mathbb{R},\quad h_{n}\underset{H^{1}}{\rightharpoonup}0,
\]
\item let $(t_{n})_{n\geq1}$, $(x_{n})_{n\geq1}$ be sequences of real
numbers, and $\bar{t},\bar{x}\in\mathbb{R}$. Then
\begin{equation}
|t_{n}|\rightarrow\infty\Longrightarrow\Vert e^{it_{n}A}\tau_{x_{n}}\psi\Vert_{L^{p}}\rightarrow0,\quad\forall2<p<\infty,\ \forall\psi\in H^{1}\label{eq:ab3}
\end{equation}
\begin{equation}
t_{n}\rightarrow\bar{t},\ x_{n}\rightarrow\pm\infty\Longrightarrow\forall\psi\in H^{1},\:\exists\varphi\in H^{1},\quad\tau_{-x_{n}}e^{it_{n}A}\tau_{x_{n}}\psi\overset{H^{1}}{\rightarrow}\varphi\label{eq:ab4}
\end{equation}
\begin{equation}
t_{n}\rightarrow\bar{t},\ x_{n}\rightarrow\bar{x}\Longrightarrow\forall\psi\in H^{1},\quad e^{it_{n}A}\tau_{x_{n}}\psi\overset{H^{1}}{\rightarrow}e^{i\bar{t}A}\tau_{\bar{x}}\psi.\label{eq:ab5}
\end{equation}
\end{itemize}
And let $(u_{n})_{n\geq1}$ be a bounded sequence in $H^{1}$. Then,
up to a subsequence, the following decomposition holds
\[
u_{n}=\sum_{j=1}^{J}e^{it_{j}^{n}A}\tau_{x_{n}^{j}}\psi_{j}+R_{n}^{J}\quad\forall J\in\mathbb{N}
\]
where
\[
t_{j}^{n}\in\mathbb{R},\;x_{j}^{n}\in\mathbb{R},\;\psi_{j}\in H^{1}
\]
are such that

\begin{itemize}
\item for any fixed $j$,
\begin{equation}
t_{j}^{n}=0\;\forall n,\quad\text{or}\quad t_{n}^{j}\overset{n\rightarrow\infty}{\rightarrow}\pm\infty\label{eq:dp1}
\end{equation}
\begin{equation}
x_{j}^{n}=0\;\forall n,\quad\text{or\quad}x_{n}^{j}\overset{n\rightarrow\infty}{\rightarrow}\pm\infty,\label{eq:dep2}
\end{equation}
\item orthogonality of the parameters:
\begin{equation}
|t_{j}^{n}-t_{k}^{n}|+|x_{j}^{n}-x_{k}^{n}|\overset{n\rightarrow\infty}{\rightarrow}\infty,\quad\forall j\neq k,\label{eq:dp3}
\end{equation}
\item decay of the reminder:
\begin{equation}
\forall\epsilon>0,\exists J\in\mathbb{N},\quad\limsup_{n\rightarrow\infty}\Vert e^{-itA}R_{n}^{J}\Vert_{L^{\infty}L^{\infty}}\leq\epsilon,\label{eq:dp4}
\end{equation}
 
\item orthogonality of the Hilbert norm:
\begin{equation}
\Vert u_{n}\Vert_{L^{2}}^{2}=\sum_{j=1}^{J}\Vert\psi_{j}\Vert_{L^{2}}^{2}+\Vert R_{n}^{J}\Vert_{L^{2}}^{2}+o_{n}(1),\quad\forall J\in\mathbb{N}\label{eq:dp5}
\end{equation}
\begin{equation}
\Vert u_{n}\Vert_{H}^{2}=\sum_{j=1}^{J}\Vert\tau_{x_{n}^{j}}\psi_{j}\Vert_{H}^{2}+\Vert R_{n}^{J}\Vert_{H}^{2}+o_{n}(1),\quad\forall J\in\mathbb{N}\label{eq:dp6}
\end{equation}
 where $(u,v)_{H}=(Au,v)$, and
\begin{equation}
\Vert u_{n}\Vert_{L^{p}}^{p}=\sum_{j=1}^{J}\Vert e^{it_{j}^{n}A}\tau_{x_{n}^{j}}\psi_{j}\Vert_{L^{p}}^{p}+\Vert R_{n}^{J}\Vert_{L^{p}}^{p}+o_{n}(1),\quad\forall2<p<2^{\star},\quad\forall J\in\mathbb{N}.\label{eq:dp7}
\end{equation}
\end{itemize}
\end{thm*}
Let us show that the self-adjoint operator $A:=-\Delta+V$ verifies
the hypothesis of the previous theorem.
\begin{prop}
\label{prop:prof_assumpt}Let $A:=-\Delta+V$. Then $A$ satisfies
the assumptions (\ref{eq:ab1}), (\ref{eq:ab2}), (\ref{eq:ab3}), (\ref{eq:ab4}), (\ref{eq:ab5}).
\end{prop}
\begin{proof}
\textbf{Assumption (\ref{eq:ab1}).} Because $V$ is non-negative,
by the Hölder inequality, (\ref{eq:delta}), and the Sobolev embedding
$H^{1}\hookrightarrow L^{2^{\star}}$, 
\begin{multline*}
\Vert u\Vert_{H^{1}}^{2}\leq(Au,u)+\Vert u\Vert_{L^{2}}=\int|\nabla u|^{2}+\int V|u|^{2}+\int|u|^{2} \\ \leq\Vert u\Vert_{H^{1}}^{2}+\Vert V\Vert_{L^{d/2}}\Vert u\Vert_{L^{2^{\star}}}^{2}
\leq(1+C_{\text{Sobolev}}\Vert V\Vert_{L^{d/2}})\Vert u\Vert_{H^{1}}^{2}.
\end{multline*}
and (\ref{eq:ab1}) holds.

\textbf{Assumption (\ref{eq:ab2}).} We have
\[
B(\tau_{x_{n}}\psi,\tau_{x_{n}}h_{n})=\int V\tau_{x_{n}}\psi\overline{\tau_{x_{n}}h_{n}}.
\]

Assume that $x_{n}\rightarrow\bar{x}\in\mathbb{R}$ and $h_{n}\underset{H^{1}}{\rightharpoonup}0$.
Notice that $B$ can also be written 
\[
B(\tau_{x_{n}}\psi,\tau_{x_{n}}h_{n})=\int(\tau_{-x_{n}}V)\psi\overline{h_{n}}.
\]
By Sobolev embedding, $h_{n}\rightharpoonup0$ weakly in $L^{2^{\star}}$.
Moreover, $\tau_{-x_{n}}V\rightarrow\tau_{-\bar{x}}V$ strongly in
$L^{d/2}$. Therefore, because $\psi\in L^{2^{\star}}$ by Sobolev
embedding again, it follows from (\ref{eq:delta}) that $B(\tau_{x_{n}}\psi,\tau_{x_{n}}h_{n})\rightarrow0$.

Now, let us assume that 
\[
x_{n}\rightarrow+\infty,\ \sup\Vert h_{n}\Vert_{H^{1}}<\infty.
\]
We fix $\epsilon>0$. By the Sobolev embedding $H^{1}\hookrightarrow L^{2^{\star}}$,
we can choose $\varLambda>0$ large enough so that 
\begin{equation}
\Vert\psi\Vert_{L^{2^{\star}}(|x|\geq\Lambda)}\leq\epsilon.\label{eq:d3r1}
\end{equation}
Because $V\in L^{d/2}$, $\varLambda$ can also be choosen large enough
so that
\begin{equation}
\Vert V\Vert_{L^{d/2}(|x|\geq\Lambda)}\leq\epsilon.\label{eq:d3r2}
\end{equation}
Then, by the Hölder inequality -- recall that $\eta$ is defined
in (\ref{eq:eta0}) as the conjugate of $2^{\star}$ -- by Sobolev
embedding and the Minkoswski inequality
\begin{multline*}
|B(\tau_{x_{n}}\psi,\tau_{x_{n}}h_{n})|\leq\Vert h_{n}\Vert_{L^{_{2^{\star}}}}\Vert V\tau_{x_{n}}\Vert_{L^{\eta}}\\
\lesssim\sup_{j\geq1}\Vert h_{j}\Vert_{H^{1}}\left(\Vert V\psi(\cdot-x_{n})\Vert_{L^{\eta}(|x-x_{n}|\geq\varLambda)}+\Vert V\psi(\cdot-x_{n})\Vert_{L^{\eta}(|x-x_{n}|\leq\varLambda)}\right).
\end{multline*}
Thus, by the Hölder inequality again, using this time (\ref{eq:eta}),
we have
\begin{equation}
|B(\tau_{x_{n}}\psi,\tau_{x_{n}}h_{n})|\lesssim\Vert V\Vert_{L^{d/2}}\Vert\psi\boldsymbol{1}_{|x|\geq\Lambda}\Vert_{L^{2^{\star}}}+\Vert V\boldsymbol{1}_{|x-x_{n}|\leq\Lambda}\Vert_{L^{d/2}}\Vert\psi\Vert_{L^{2^{\star}}}.\label{eq:d3r3}
\end{equation}
Now, let $n_{0}$ be large enough so that for all $n\geq n_{0}$,
$x_{n}\geq2\varLambda$. Then, for all $n\geq n_{0}$ 
\[
|x-x_{n}|\leq\varLambda\Rightarrow|x|\geq\varLambda
\]
and, for all $n\geq n_{0}$ we get by (\ref{eq:d3r1}), (\ref{eq:d3r2}),
(\ref{eq:d3r3})
\[
|B(\tau_{x_{n}}\psi,\tau_{x_{n}}h_{n})|\lesssim\left(\epsilon\Vert V\Vert_{L^{\delta}}+\epsilon\Vert\psi\Vert_{L^{2^{\star}}}\right)
\]
so (\ref{eq:ab2}) holds.

\textbf{Assumption (\ref{eq:ab3}).} The same proof as in \cite{AsymptAn}
holds: it is an immediate consequence of the pointwise dispersive
estimate (\ref{eq:dispest}) and the translation invariance of the
$L^{p}$ norms. Notice that the estimate
\begin{equation}
\Vert e^{itA}f\Vert_{H^{1}}\lesssim\Vert f\Vert_{H^{1}},\label{eq:subiso}
\end{equation}
which is usefull to close the density argument of this previous paper,
generalizes to dimensions $d\geq2$ because, as $V$ is positive and
in $L^{d/2}$, by the Hölder inequality together with the Sobolev
embedding $H^{1}\hookrightarrow L^{2^{\star}}$ we get
\begin{multline}
\Vert\nabla f\Vert_{L^{2}}^{2}\leq\Vert(-\Delta+V)^{\frac{1}{2}}f\Vert_{L^{2}}^{2}=\int|\nabla u|^{2}+\int V|u|^{2} \\ \leq\Vert f\Vert_{H^{1}}^{2}+\Vert V\Vert_{L^{d/2}}\Vert u\Vert_{L^{2^{\star}}}^{2}\lesssim\Vert f\Vert_{H^{1}}^{2},
\end{multline}
from which (\ref{eq:subiso}) follows because $e^{itA}$ commute with
$(-\Delta+V)^{\frac{1}{2}}$ and is an isometry on $L^{2}$.

\textbf{Assumption (\ref{eq:ab4}).} We will show that
\[
t_{n}\rightarrow\bar{t},\:x_{n}\rightarrow+\infty\;\Rightarrow\:\Vert\tau_{-x_{n}}e^{it_{n}(-\Delta+V)}\tau_{x_{n}}\psi-e^{-i\bar{t}\Delta}\psi\Vert_{H^{1}}\rightarrow0
\]
and (\ref{eq:ab4}) will hold with $\varphi=e^{-i\bar{t}\Delta}\psi$.
As remarked in \cite{AsymptAn}, it is sufficent to show that 
\begin{equation}
\Vert e^{it_{n}(-\Delta+V)}\tau_{x_{n}}\psi-e^{-it_{n}\Delta}\tau_{x_{n}}\psi\Vert_{H^{1}}\rightarrow0.\label{eq:l1}
\end{equation}

Notice $e^{-it\Delta}\tau_{x_{n}}\psi-e^{it(-\Delta+V)}\tau_{x_{n}}\psi$
is a solution of the following linear Schrödinger equation with zero
initial data
\[
i\partial_{t}u-\Delta u+Vu=Ve^{-it\Delta}\tau_{x_{n}}\psi.
\]
Therefore, by the inhomogenous Strichartz estimates, as $(2,2^{\star})$
is admissible in dimension $d\geq3$ with dual exponent $(2,\eta)$,
and because the translation operator commutes with $e^{-it\Delta}$,
we have for $n$ large enough so that $t_{n}\in(0,\bar{t}+1)$ 
\begin{multline*}
\Vert e^{it_{n}(-\Delta+V)}\tau_{x_{n}}\psi-e^{-it_{n}\Delta}\tau_{x_{n}}\psi\Vert_{L^{2}}\\
\leq\Vert e^{it(-\Delta+V)}\tau_{x_{n}}\psi-e^{-it\Delta}\tau_{x_{n}}\psi\Vert_{L^{\infty}(0,\bar{t}+1)L^{2}}\leq\Vert Ve^{-it\Delta}\tau_{x_{n}}\psi\Vert_{L^{2}(0,\bar{t}+1)L^{\eta}}\\
=\Vert(\tau_{-x_{n}}V)e^{-it\Delta}\psi\Vert_{L^{2}(0,\bar{t}+1)L^{\eta}}\leq(\bar{t}+1)^{1/2}\Vert(\tau_{-x_{n}}V)e^{-it\Delta}\psi\Vert_{L^{\infty}(0,\bar{t}+1)L^{\eta}}.
\end{multline*}
Hence, estimating in the same manner the gradient of these quantities,
it is sufficient to obtain (\ref{eq:l1}) to show that, as $n$ goes
to infinity
\begin{equation}
\Vert(\tau_{-x_{n}}V)e^{-it\Delta}\psi\Vert_{L^{\infty}(0,\bar{t}+1)W^{1,\eta}}\rightarrow0.\label{eq:vw11}
\end{equation}

Let us fix $\epsilon>0$. By Sobolev embedding in $L^{2^{\star}}$,
because $e^{-it\Delta}\psi\in C([0,\bar{t}+1],H^{1})$ and using the
compacity in time, there exists $\varLambda>0$ such that 
\begin{equation}
\Vert e^{-it\Delta}\psi\Vert_{L^{\infty}(0,\bar{t}+1)L^{2^{\star}}(|x|\geq\varLambda)}\leq\epsilon.\label{eq:d3ra21}
\end{equation}
On the other hand, as $V\in L^{d/2}$, $\varLambda$ can also be taken
large enough so that
\begin{equation}
\Vert V\Vert_{L^{d/2}(|x|\geq\Lambda)}\leq\epsilon.\label{eq:d3ra22}
\end{equation}
Let $n_{0}$ be large enough so that for all $n\geq n_{0}$, $x_{n}\geq2\varLambda$.
Then, for $n\geq n_{0}$ 
\[
|x+x_{n}|\leq\varLambda\Rightarrow|x|\geq\varLambda
\]
and for all $t\in(0,\bar{t}+1)$ and all $n\geq n_{0}$ we obtain,
by Minkowski inequality, Hölder inequality together with (\ref{eq:d3ra21})
and (\ref{eq:d3ra22}), and Sobolev embedding
\begin{multline*}
\Vert(\tau_{-x_{n}}V)e^{-it\Delta}\psi\Vert_{L^{\eta}}  \\ \ \leq\Vert V(\cdot+x_{n})e^{-it\Delta}\psi\Vert_{L^{\eta}(|x+x_{n}|\geq\varLambda)}+\Vert V(\cdot+x_{n})e^{-it\Delta}\psi\Vert_{L^{\eta}(|x+x_{n}|\leq\varLambda)}  \\
\ \leq\epsilon\Vert e^{-it\Delta}\psi\Vert_{L^{\infty}(0,\bar{t}+1)L^{2^{\star}}}+\epsilon\Vert V\Vert_{L^{\delta}}\lesssim\epsilon(\Vert e^{-it\Delta}\psi\Vert_{L^{\infty}(0,\bar{t}+1)H^{1}}+\Vert V\Vert_{L^{\delta}}),
\end{multline*}
thus 
\[
\Vert(\tau_{-x_{n}}V)e^{-it\Delta}\psi\Vert_{L^{\infty}(0,\bar{t}+1)L^{\eta}}\rightarrow0.
\]
 With the same argument, because $\nabla V\in L^{d/2}$, we have
\[
\Vert\nabla(\tau_{-x_{n}}V)e^{-it\Delta}\psi\Vert_{L^{\infty}(0,\bar{t}+1)L^{\eta}}\rightarrow0.
\]
Hence, to obtain (\ref{eq:vw11}), it only remains to show that
\begin{equation}
\Vert\tau_{-x_{n}}V\nabla(e^{-it\Delta}\psi)\Vert_{L^{\infty}(0,\bar{t}+1)L^{\eta}}\rightarrow0.\label{eq:assab4rem}
\end{equation}
To this purpose, let $\tilde{\psi}$ be a $C^{\infty}$, compactly
supported function such that 
\[
\Vert\psi-\tilde{\psi}\Vert_{H^{1}}\leq\epsilon.
\]
Notice that, by (\ref{eq:eta0}) we have
\[
\frac{1}{\eta}=\frac{1}{2}+\frac{1}{d},
\]
hence, by Minkowski and Hölder inequalities,
\begin{align}
\Vert\tau_{-x_{n}}V\nabla(e^{-it\Delta}\psi)\Vert_{L^{\eta}} & \leq\Vert\tau_{-x_{n}}V\nabla(e^{-it\Delta}\tilde{\psi})\Vert_{L^{\eta}}+\Vert\tau_{-x_{n}}V\nabla(e^{-it\Delta}(\psi-\tilde{\psi}))\Vert_{L^{\eta}}\nonumber \\
 & \leq\Vert\tau_{-x_{n}}V\nabla(e^{-it\Delta}\tilde{\psi})\Vert_{L^{\eta}}+\Vert V\Vert_{L^{d}}\Vert\nabla(e^{-it\Delta}(\psi-\tilde{\psi}))\Vert_{L^{2}}\nonumber \\
 & \leq\Vert\tau_{-x_{n}}V\nabla(e^{-it\Delta}\tilde{\psi})\Vert_{L^{\eta}}+\epsilon\Vert V\Vert_{L^{d}},\label{eq:rd3fas}
\end{align}
where $V\in L^{d}$ because of the (critical) Sobolev embedding $W^{1,d/2}(\mathbb{R}^{d})\hookrightarrow L^{d}(\mathbb{R}^{d})$.

Then, because $\nabla(e^{-it\Delta}\tilde{\psi})\in H^{1}$, 
\[
\Vert\tau_{-x_{n}}V\nabla(e^{-it\Delta}\tilde{\psi})\Vert_{L^{\infty}(0,\bar{t}+1)L^{\eta}}
\]
 can be estimated as $\Vert(\tau_{-x_{n}}V)e^{-it\Delta}\psi\Vert_{L^{\infty}(0,\bar{t}+1)L^{\eta}}$,
hence (\ref{eq:assab4rem}) follows from (\ref{eq:rd3fas}) and the
assumption is verified.

\textbf{Assumption (\ref{eq:ab5}).} It is a consequence of (\ref{eq:subiso}),
the Lebesgue's dominated convergence theorem and the continuity of
$t\in\mathbb{R}\longrightarrow e^{itA}\tau_{\bar{x}}\psi\in H^{1}$
with the exact same proof as in \cite{AsymptAn}.
\end{proof}

\subsection{Non linear profiles }

Similarly to \cite{AsymptAn}, we will now see that for a data which
escapes to infinity, the solutions of (\ref{eq:nlsv}) are the same
as these of the homogeneous equation ($V=0$), in the sense given
by the three next Propositions:
\begin{prop}
\label{prop:decay_lin}Let $\psi\in H^{1}$, $(x_{n})_{n\geq1}\in\mathbb{R}^{\mathbb{N}}$
be such that $|x_{n}|\rightarrow\infty$. Then, up to a subsequence
\begin{equation}
\Vert e^{-it\Delta}\tau_{x_{n}}\psi-e^{-it(\Delta-V)}\tau_{x_{n}}\psi\Vert_{L^{p}L^{r}}\rightarrow0\label{eq:dec1}
\end{equation}
as $n\rightarrow\infty$.
\end{prop}
\begin{proof}
We assume for example $x_{n}\rightarrow+\infty$.

By the dispersive estimate and a density argument, the same proof
as in \cite{AsymptAn} gives

\begin{equation}
\sup_{n\in\mathbb{N}}\Vert e^{it(-\Delta+V)}\tau_{x_{n}}\psi\Vert_{L^{p}(T,\infty)L^{r}}\rightarrow0\label{eq:loctemp}
\end{equation}
as $T\rightarrow\infty$. We are therefore reduced to show that for
$T>0$ fixed

\[
\Vert e^{-it\Delta}\tau_{x_{n}}\psi-e^{it(-\Delta+V)}\tau_{x_{n}}\psi\Vert_{L^{p}(0,T)L^{r}}\rightarrow0
\]
as $n\rightarrow\infty$. Let us pick $\epsilon>0$. The differene
$e^{-it\Delta}\tau_{x_{n}}\psi-e^{it(-\Delta+V)}\tau_{x_{n}}\psi$
is a solution of the following linear Schrödinger equation with zero
initial data
\[
i\partial_{t}u-\Delta u+Vu=Ve^{-it\Delta}\tau_{x_{n}}\psi.
\]
So, by the inhomogenous Strichartz estimate (\ref{eq:st4-1}) 
\begin{eqnarray*}
\Vert e^{-it\Delta}\tau_{x_{n}}\psi-e^{it(-\Delta+V)}\tau_{x_{n}}\psi\Vert_{L_{t}^{p}(0,T)L^{r}} & \lesssim & \Vert Ve^{-it\Delta}\tau_{x_{n}}\psi\Vert_{L_{t}^{\gamma'}(0,T)L^{\eta}}\\
 & \lesssim & T^{\frac{1}{\gamma'}}\Vert Ve^{-it\Delta}\tau_{x_{n}}\psi\Vert_{L^{\infty}(0,T)L^{\eta}}\\
 &  & =T^{\frac{1}{\gamma'}}\Vert(\tau_{-x_{n}}V)e^{-it\Delta}\psi\Vert_{L^{\infty}(0,T)L^{\eta}}
\end{eqnarray*}
because the translation operator $\tau_{x_{n}}$ commutes with the
propagator $e^{-it\Delta}$. But 
\[
\Vert(\tau_{-x_{n}}V)e^{-it\Delta}\psi\Vert_{L^{\infty}(0,T)L^{\eta}}\underset{n\rightarrow\infty}{\longrightarrow}0
\]
as seen in the proof of Proposition \ref{prop:prof_assumpt}, point
(\ref{eq:ab4}).
\end{proof}
\begin{prop}
\label{prop:decay_duhamel}Let $\psi\in H^{1}$, $(x_{n})_{n\geq1}\in\mathbb{R}^{\mathbb{N}}$
be such that $|x_{n}|\rightarrow\infty$, $U\in C(H^{1})\cap L^{p}L^{r}$
be the unique solution to the homogeneous equation
\[
i\partial_{t}u+\Delta u=u|u|^{\alpha}
\]
with initial data $\psi$, and $U_{n}(t,x):=U(t,x-x_{n})$. Then,
up to a subsequence

\begin{equation}
\Vert\int_{0}^{t}e^{-i(t-s)\Delta}\left(U_{n}|U_{n}|^{\alpha}\right)(s)ds-\int_{0}^{t}e^{-i(t-s)(\Delta-V)}\left(U_{n}|U_{n}|^{\alpha}\right)(s)ds\Vert_{L^{p}L^{r}}\rightarrow0\label{eq:dec2}
\end{equation}
as $n\rightarrow\infty$. 
\end{prop}
\begin{proof}
In the exact same way as in \cite{AsymptAn}, inhomogenous Strichartz
estimates, and the pointwise dispersive estimate together with Hardy-Littlewood-Sobolev
inequality leads
\begin{equation}
\sup_{n\in\mathbb{N}}\Vert\int_{0}^{t}e^{-i(t-s)(\Delta-V)}\left(U_{n}|U_{n}|^{\alpha}\right)(s)ds\Vert_{L^{p}([T,\infty))L^{r}}\rightarrow0\label{eq:dec2-2}
\end{equation}
as $T$ goes to infinity. Thus; it remains to show that for $T>0$
fixed, 
\[
\Vert\int_{0}^{t}e^{-i(t-s)\Delta}\left(U_{n}|U_{n}|^{\alpha}\right)ds-\int_{0}^{t}e^{-i(t-s)(\Delta-V)}\left(U_{n}|U_{n}|^{\alpha}\right)ds\Vert_{L^{p}(0,T)L^{r}}\rightarrow0
\]
as $n\rightarrow\infty$. The difference
\[
\int_{0}^{t}e^{-i(t-s)\Delta}\left(U_{n}|U_{n}|^{\alpha}\right)ds-\int_{0}^{t}e^{-i(t-s)(\Delta-V)}\left(U_{n}|U_{n}|^{\alpha}\right)ds
\]
 is the solution of the following linear Schrödinger equation, with
zero initial data
\[
i\partial_{t}u-\Delta u+Vu=V\int_{0}^{t}e^{-i(t-s)\Delta}\left(U_{n}|U_{n}|^{\alpha}\right)ds.
\]
Hence, by the Strichartz estimate (\ref{eq:st4-1})
\begin{multline*}
\Vert\int_{0}^{t}e^{-i(t-s)\Delta}\left(U_{n}|U_{n}|^{\alpha}\right)ds-\int_{0}^{t}e^{-i(t-s)(\Delta-V)}\left(U_{n}|U_{n}|^{\alpha}\right)ds\Vert_{L^{p}(0,T)L^{r}}\\
\lesssim\Vert V\int_{0}^{t}e^{-i(t-s)\Delta}\left(U_{n}|U_{n}|^{\alpha}\right)ds\Vert_{L^{\gamma'}(0,T)L^{\eta}}\\
\lesssim T^{\frac{1}{\gamma'}}\Vert(\tau_{-x_{n}}V)\int_{0}^{t}e^{-i(t-s)\Delta}\left(U|U|^{\alpha}\right)ds\Vert_{L^{\infty}(0,T)L^{\eta}}.
\end{multline*}
But $\int_{0}^{t}e^{-i(t-s)\Delta}(U|U|^{\alpha})ds\in C([0,T],H^{1})$,
so by Sobolev embedding in $L^{2^{\star}}$ and compacity in time
there exists $\varLambda>0$ such that 
\[
\Vert\int_{0}^{t}e^{-i(t-s)\Delta}\left(U|U|^{\alpha}\right)ds\Vert_{L^{\infty}(0,T)L^{2^{\star}}(|x|\geq\varLambda)}\leq\epsilon
\]
therefore
\[
\Vert(\tau_{-x_{n}}V)\int_{0}^{t}e^{-i(t-s)\Delta}\left(U|U|^{\alpha}\right)ds\Vert_{L^{\infty}(0,T)L^{\eta}}\underset{n\rightarrow\infty}{\longrightarrow}0
\]
in the same way as in the proof of Proposition \ref{prop:prof_assumpt},
point (\ref{eq:ab4}) .
\end{proof}
\begin{prop}
\label{prop:tx_inf}Let $\psi\in H^{1}$, $(x_{n})_{n\geq1},\,(t_{n})_{n\geq1}\in\mathbb{R}^{\mathbb{N}}$
be such that $|x_{n}|\rightarrow\infty$ and $t_{n}\rightarrow\pm\infty$,
$U$ be a solution to the homogeneous equationsuch that
\[
\Vert U(t)-e^{-it\Delta}\psi\Vert_{H^{1}}\underset{t\rightarrow\pm\infty}{\longrightarrow}0
\]
and $U_{n}(t,x):=U(t-t_{n},x-x_{n})$. Then, up to a subsequence
\begin{equation}
\Vert e^{-i(t-t_{n})\Delta}\tau_{x_{n}}\psi-e^{-i(t-t_{n})(\Delta-V)}\tau_{x_{n}}\psi\Vert_{L^{p}L^{r}}\rightarrow0\label{eq:dec1-1}
\end{equation}
and

\begin{equation}
\Vert\int_{0}^{t}e^{-i(t-s)\Delta}\left(U_{n}|U_{n}|^{\alpha}\right)ds-\int_{0}^{t}e^{-i(t-s)(\Delta-V)}\left(U_{n}|U_{n}|^{\alpha}\right)ds\Vert_{L^{p}L^{r}}\rightarrow0\label{eq:dec2-1}
\end{equation}
as $n\rightarrow\infty$. 
\end{prop}
\begin{proof}
The proof is the same as for Proposition \ref{prop:decay_lin} and
Proposition \ref{prop:decay_duhamel}, decomposing the time interval
in $\left\{ |t-t_{n}|>T\right\} $ and his complementary.
\end{proof}
Finaly, we will need the following Proposition of non linear scattering:
\begin{prop}
\label{prop:lin_scatt}Let $\varphi\in H^{1}$. Then there exists
$W_{\pm}\in C(H^{1})\cap L_{\mathbb{R}^{\pm}}^{p}L^{r}$, solution
of (\ref{eq:nlsv}) such that
\begin{equation}
\Vert W_{\pm}(t,\cdot)-e^{-it(\Delta-V)}\varphi\Vert_{H^{1}}\underset{t\rightarrow\pm\infty}{\longrightarrow}0\label{eq:nl1}
\end{equation}
moreover, if $t_{n}\rightarrow\mp\infty$ and 
\begin{equation}
\varphi_{n}=e^{-it_{n}(\Delta-V)}\varphi,\quad W_{\pm,n}(t)=W_{\pm}(t-t_{n})\label{eq:nl2}
\end{equation}
then
\begin{equation}
W_{\pm,n}(t)=e^{-it(\Delta-V)}\varphi_{n}+\int_{0}^{t}e^{-i(t-s)(\Delta-V)}(W_{\pm,n}|W_{\pm,n}|^{\alpha})(s)ds+f_{\pm,n}(t)\label{eq:nl3}
\end{equation}
where
\begin{equation}
\Vert f_{\pm,n}\Vert_{L_{\mathbb{R}^{\pm}}^{p}L^{r}}\underset{n\rightarrow\infty}{\longrightarrow0}.\label{eq:nl4}
\end{equation}
\end{prop}
\begin{proof}
The same proof as \cite{MR213495011}, Proposition 3.5, holds, as
it involves only the analogous Strichartz estimates.
\end{proof}

\subsection{Conclusion}

The \thmref{critsol} is now a consequence of the linear profile decomposition
together with the nonlinear profiles results of Propositions \propref{decay_lin},
\propref{decay_duhamel}, \propref{tx_inf}, \propref{lin_scatt},
perturbative result of Proposition \ref{pert3} and Strichartz inequalities
of Proposition \ref{strichartz}, in the exact same way as in \cite{AsymptAn},
Section 5.

\section{Extinction of the critical solution}

The aim of this section is to prove the following ridity theorem
\begin{thm}
\label{thm:rigidity}There is no non-trivial compact-flow solution
of (\ref{eq:nlsv}). 
\end{thm}
By compact flow solution, we mean here a solution $u$ with a relatively
compact flow $\left\{ u_{c}(t),\ t\geq0\right\} $ in $H^{1}$. Our
key tool will be the following Morawetz identity -- or virial computation:

\begin{lem}
Let $u\in C(H^{1})$ be a solution of (\ref{eq:nlsv}) and $\chi\in C^{\infty}(\mathbb{R}^{d})$
be a smooth function. Then

\begin{equation}
\partial_{t}\int\chi|u|^{2}=2\text{Im}\int\nabla\chi\cdot\nabla u\bar{u}\label{eq:vir1}
\end{equation}
\begin{multline}
\partial_{t}^{2}\int\chi|u|^{2}=4\int(D^{2}\chi\nabla u,\nabla u)+\frac{2}{\alpha+2}\int\Delta\chi|u|^{\alpha+2} \\
-2\int\nabla\chi\cdot\nabla V|u|^{2}-\int\Delta^{2}\chi|u|^{2}.\label{eq:vir2}
\end{multline}
\end{lem}
In the case of a repulsive potential, taking the weight $\chi=|x|^{2}$
gives the result by a classical argument, as all the terms, and in
particular
\begin{equation}
\int\nabla\chi\cdot\nabla V|u|^{2}\label{eq:termechiant}
\end{equation}
have the right sign. However, with a non-repulsive potential, this
straightforward choose of weight does not permit to conclude because
(\ref{eq:termechiant}) is no signed anymore.

However, in our framework of the sum of two repulsive potentials verifying
the convexity assumptions (G1)-(G2)-(G3), we are able to construct
a family of weights that have the right behavior and for which the
non-negative part of (\ref{eq:termechiant}) can be made small enough.
The idea is to construct it in such a way that $\nabla\chi$ is almost
orthogonal to the line $\mathcal{R}$ containing the trapped trajectories.
More precisely, we would like to take as a weight
\[
|x-\boldsymbol{c}|+|x+\boldsymbol{c}|,
\]
where $\boldsymbol{c}$ is such that $(-\boldsymbol{c},\boldsymbol{c})\subset\mathcal{R}$
and will be sent to infinity. 

The smallness of the non-negative part (\ref{eq:termechiant}) will
be a consequence of the following lemma, where $\Theta_{1}$ and $\Theta_{2}$
have to be tought as level surfaces of $V_{1}$ and $V_{2}$. The
assumptions (2) and (3) of the lemma correponds to assumptions (G2)
and (G3). In the following, $n$ is choosen as the outward-pointing
normal to $\Theta_{1}$ and $\Theta_{2}$. 
\begin{lem}
\label{lem:poids}Let $\alpha>0$, $R\in C^{0}([A,+\infty[,\mathbb{R}_{+})$
be such that $R(c)/c\longrightarrow0\text{ as }c\longrightarrow+\infty$
and, for all $c\geq A$, $(\Theta_{1})(c)$, $(\Theta_{2})(c)$ be
two families of smooth convex subsets of $\mathbb{R}^{d}$. We assume
that, for all $c\geq A$ and any elements $\Theta_{1},\Theta_{2}$
of $(\Theta_{1})(c),(\Theta_{2})(c)$
\begin{enumerate}
\item $\Theta_{1}$ and $\Theta_{2}$ are contained in $B(0,R(c))$,
\item in the non star-shaped region $\left\{ x\in\partial(\Theta_{1}\cup\Theta_{2}),\ x\cdot n(x)<0\right\} $,
the eigenvalues of the second fundamental forms of $\partial\Theta_{1}$
and $\partial\Theta_{2}$ are bounded below by $\alpha$,
\item the trapped ray associated with $\mathbb{R}^{d}\backslash(\Theta_{1}\cup\Theta_{2})$
is a segment of the line $\left\{ x_{2}=\cdots=x_{d}=0\right\} $.
\end{enumerate}
Let $\boldsymbol{c}:=(c,0,\cdots,0)$. Then, for any elements $\Theta_{1},\Theta_{2}$
of $(\Theta_{1})(c),(\Theta_{2})(c)$ and $x\in\partial(\Theta_{1}\cup\Theta_{2})$,
we have as $c\longrightarrow+\infty$ 
\[
\left(\frac{x-\boldsymbol{c}}{|x-\boldsymbol{c}|}+\frac{x+\boldsymbol{c}}{|x+\boldsymbol{c}|}\right)\cdot n(x)\geq O(\frac{R(c)^{4}}{c^{4}}).
\]
\end{lem}
\begin{proof}
For $x\in B(0,R)$, let us denote $x=(x_{1},\tilde{x})$ with $\tilde{x}\in\mathbb{R}^{d-1}$.
Remark that 
\[
|x+\boldsymbol{c}|=c+x_{1}+\frac{1}{2c}|\tilde{x}|^{2}+O(\frac{R^{4}}{c^{3}})
\]
and therefore
\begin{equation}
\frac{x-\boldsymbol{c}}{|x-\boldsymbol{c}|}+\frac{x+\boldsymbol{c}}{|x+\boldsymbol{c}|}=\frac{1}{|x-\boldsymbol{c}||x+\boldsymbol{c}|}\left(2c(0,\tilde{x})+x\frac{|\tilde{x}|^{2}}{c}+O(\frac{R^{4}}{c^{2}})\right).\label{eq:lem1dl}
\end{equation}

Notice that, in the star-shaped region $\left\{ x\in\partial(\Theta_{1}\cup\Theta_{2})(c),\ x\cdot n(x)\geq0\right\} $,
(\ref{eq:lem1dl}) together with the fact that $\tilde{x}\cdot n\geq0$
by convexity of the obstacles, and noticing that 
\begin{equation}
|x-\boldsymbol{c}||x+\boldsymbol{c}|\gtrsim c^{2}\label{eq:finlemgeqc}
\end{equation}
by the hypothesis $R(c)/c\longrightarrow0$, gives the result.

Let us now consider $x$ in the more intricate non star-shaped region
\[
\left\{ x\in\partial(\Theta_{1}\cup\Theta_{2})(c),\ x\cdot n(x)<0\right\} .
\]
On $\partial\Theta_{i}$, $n$ is near the trapped ray associated
with $\mathbb{R}^{d}\backslash(\Theta_{1}\cup\Theta_{2})$ of the
form 
\[
n(x)=(\pm\frac{x_{1}}{|x_{1}|},0,\cdots,0)+(0,\lambda_{2}x_{2},\cdots,\lambda_{d}x_{d})+O(|\tilde{x}|^{2})
\]
with $\lambda_{k}>0$. And thus
\begin{gather}
\left(2c(0,\tilde{x})+x\frac{|\tilde{x}|^{2}}{c}\right)\cdot n(x)\geq\left(2c\min\lambda_{k}-\frac{C}{c}\right)|\tilde{x}|^{2}+O(|\tilde{x}|^{2}).\label{eq:minlamd}
\end{gather}
Because of the uniform convexity assumption of $\Theta_{1,2}$ in
the non star-shaped region (assumption (2) of the lemma), $\min\lambda_{k}$
is bounded below, uniformly in $c$, by a strictly positive, universal
constant. Hence, by (\ref{eq:minlamd}), there exists $\rho\geq0$
and $D_{1}>0$ such that, for every $c>D_{1}$ we have 
\begin{equation}
|\tilde{x}|\leq\rho\implies\left(2c(0,\tilde{x})+x\frac{|\tilde{x}|^{2}}{c}\right)\cdot n(x)\geq0.\label{eq:l1petdelt}
\end{equation}

On the other hand, there exists $\epsilon_{0}>0$ such that, for all
$x\in\partial(\Theta_{1}\cup\Theta_{2})$,
\[
|\tilde{x}|\geq\rho\implies(0,\tilde{x})\cdot n(x)\geq\epsilon_{0}.
\]
Notice that the uniformity in $c$ is a consequence of assumption
(2) again. Hence, if $|\tilde{x}|\geq\rho$
\[
\left(2c(0,\tilde{x})+x\frac{|\tilde{x}|^{2}}{c}\right)\cdot n(x)\geq2c\epsilon_{0}-\frac{C}{c},
\]
and therefore, there exists $D_{2}>0$ such that, if $c>D_{2}$
\begin{equation}
|\tilde{x}|\geq\rho\implies\left(2c(0,\tilde{x})+x\frac{|\tilde{x}|^{2}}{c}\right)\cdot n(x)\geq0.\label{eq:l1grddelt}
\end{equation}

Combining (\ref{eq:lem1dl}), (\ref{eq:l1petdelt}), (\ref{eq:l1grddelt}),
and (\ref{eq:finlemgeqc}) gives the result.
\end{proof}
We are now in position to prove the rigidity theorem:
\begin{proof}[Proof of \thmref{rigidity}.]
By contradiction, let $u\neq0$ be a solution of (\ref{eq:nlsv})
with a relatively compact flow $\left\{ u(t)\text{,\ }t\in\mathbb{R}\right\} $
in $H^{1}$ . 

We choose a system of coordinates such that $\mathcal{R}=\left\{ x_{2}=\cdots=x_{d}=0\right\} $.
Let $c>0$ and $\boldsymbol{c}:=(c,0,\dots,0)$. We would like to
take
\begin{equation}
|x-\boldsymbol{c}|+|x+\boldsymbol{c}|\label{eq:x_c_naiv}
\end{equation}
 as a weight. However, because of the singularities in $\pm\boldsymbol{c}$,
it is not smooth and we cannot use it explicitly. Therefore, we take
instead
\[
\chi_{c}(x):=\left(|x-\boldsymbol{c}|+|x+\boldsymbol{c}|\right)\psi(\frac{x}{c/4}),
\]
where $\psi\in C^{\infty}$ is such that $\psi(x)=1$ for $|x|\leq1$
and $\psi(x)=0$ for $|x|\geq2$. The idea is that now $\chi_{c}$
is smooth, and it coincides with (\ref{eq:x_c_naiv}) in $B(0,c/4)$,
but as $c$ will be sent to infinity, the part outside this ball will
not be seen by compact flow solutions. Let us denote 
\[
z(t)=\int\chi_{c}|u|^{2}.
\]
By (\ref{eq:vir1}), the Cauchy-Schwarz inequality and the conservation
of mass and energy 
\begin{equation}
|z'(t)|\leq\sqrt{CE(u)M(u)}.\label{eq:prig1}
\end{equation}
Moreover, (\ref{eq:vir2}) writes
\begin{multline}
z''(t)=4\int(D^{2}\chi_{c}\nabla u,\nabla u)+\frac{2}{\alpha+2}\int\Delta\chi_{c}|u|^{\alpha+2}\\
-2\int\nabla\chi_{c}\cdot\nabla V|u|^{2}-\int\Delta^{2}\chi_{c}|u|^{2}.\label{eq:zsec}
\end{multline}
Let us write down $D^{2}\chi_{c}$, $\Delta\chi_{c}$ and $\Delta^{2}\chi_{c}$.
To this purpose, let 
\[
\chi_{c}^{-}(x):=|x-\boldsymbol{c}|\psi(\frac{x}{c/4}),\hspace{1em}\chi_{c}^{+}(x):=|x+\boldsymbol{c}|\psi(\frac{x}{c/4}),
\]
in such a way that $\chi_{c}=\chi_{c}^{+}+\chi_{c}^{-}$. We have
\begin{equation}
\Delta\chi_{c}^{\pm}(x)=\frac{d-1}{|x\pm\boldsymbol{c}|}\psi(\frac{x}{c/4})+\frac{8}{c}\frac{x\pm\boldsymbol{c}}{|x\pm\boldsymbol{c}|}\cdot\nabla\psi(\frac{x}{c/4})+\frac{16}{c^{2}}|x\pm\boldsymbol{c}|\Delta\psi(\frac{x}{c/4}),\label{eq:lap_m}
\end{equation}
\begin{align}
D^{2}\chi_{c}^{\pm}(x) & =\frac{1}{|x\pm\boldsymbol{c}|}\Big(\text{Id}-\frac{(x\pm\boldsymbol{c})(x\pm\boldsymbol{c})^{\mathrm{t}}}{|x\pm\boldsymbol{c}|^{2}}\Big)\psi(\frac{x}{c/4})+\frac{4}{c}\frac{x\pm\boldsymbol{c}}{|x\pm\boldsymbol{c}|}\big(\nabla\psi(\frac{x}{c/4})\big)^{\mathrm{t}}\label{eq:hess_m}\\
 & \hspace{1em}+\frac{4}{c}\nabla\psi(\frac{x}{c/4})\big(\frac{x\pm\boldsymbol{c}}{|x\pm\boldsymbol{c}|}\big)^{\mathrm{t}}+\frac{16}{c^{2}}|x\pm\boldsymbol{c}|D^{2}\psi(\frac{x}{c/4}),\nonumber 
\end{align}
\begin{align}
\Delta^{2}\chi_{c}^{\pm}(x) & =-\frac{(d-1)(d-3)}{|x\pm\boldsymbol{c}|^{3}}\psi(\frac{x}{c/4})-\frac{16(d-1)}{c}\frac{x\pm\boldsymbol{c}}{|x\pm\boldsymbol{c}|}\cdot\nabla\psi(\frac{x}{c/4})\label{eq:bil_m} \\
 & \hspace{1em}+\frac{32}{c^{2}}\frac{d+1}{|x\pm\boldsymbol{c}|}\Delta\psi(\frac{x}{c/4})+\frac{256}{c^{3}}\frac{x\pm\boldsymbol{c}}{|x\pm\boldsymbol{c}|}\cdot\nabla\big(\Delta\psi(\frac{x}{c/4})\big) \nonumber \\
 & \hspace{1em}-\frac{64}{c^{2}}\frac{1}{|x\pm\boldsymbol{c}|^{3}}\big(D^{2}\psi(\frac{x}{c/4})(x\pm\boldsymbol{c}),x\pm\boldsymbol{c}\Big)\nonumber
  +\frac{256}{c^{4}}|x\pm\boldsymbol{c}|\Delta^{2}\psi(\frac{x}{c/4}),\nonumber 
\end{align}
where $X^{\mathrm{t}}$ denote the transpose of the column vector
$X\in\mathbb{R}^{d}$. Observe that
\[
x\in\text{Supp}\psi(\frac{\cdot}{c/4})\implies|x\pm c|\geq\frac{c}{2},
\]
therefore, by (\ref{eq:lap_m}), (\ref{eq:hess_m}) and (\ref{eq:bil_m})
\begin{equation}
|\Delta\chi_{c}|+|D^{2}\chi_{c}|+|\Delta^{2}\chi_{c}|\lesssim\frac{1}{c}.\label{eq:mor_dec_c}
\end{equation}
In addition, as $\left\{ u(t)\text{,\ }t\in\mathbb{R}\right\} $ is
relatively compact in $H^{1}$ and by Sobolev embedding in $L^{\alpha+2}$
-- recall that, by assumption (\ref{eq:intercrit}), we are in particular
in the subcritical regime --, we have 
\begin{equation}
\sup_{t\in\mathbb{R}}\left(\Vert u\Vert_{L^{\alpha+2}(|x|\geq c/4)}+\Vert u\Vert_{L^{2}(|x|\geq c/4)}+\Vert\nabla u\Vert_{L^{2}(|x|\geq c/4)}\right)=\epsilon(c),\label{eq:mass_infini}
\end{equation}
where $\epsilon(c)\longrightarrow0$ as $c\longrightarrow+\infty$.
Therefore (\ref{eq:zsec}) together with (\ref{eq:mor_dec_c}) and
(\ref{eq:mass_infini}) yields
\begin{multline}
z''(t)=\int_{B(0,c/4)}4(D^{2}\chi_{c}\nabla u,\nabla u)+\frac{2}{\alpha+2}\Delta\chi_{c}|u|^{\alpha+2}\\ -\Delta^{2}\chi_{c}|u|^{2}-2\int_{\mathbb{R}^{d}}\nabla\chi_{c}\cdot\nabla V|u|^{2}
+\frac{1}{c}\epsilon(c).\label{eq:with_epsc}
\end{multline}
Now, in $B(0,c/4)$, $\chi_{c}$ coincides with (\ref{eq:x_c_naiv})
and:
\begin{gather*}
\Delta\chi_{c}^{\pm}=\frac{d-1}{|x\pm\boldsymbol{c}|},\ D^{2}\chi_{c}^{\pm}=\frac{1}{|x\pm\boldsymbol{c}|}\Big(\text{Id}-\frac{(x\pm\boldsymbol{c})(x\pm\boldsymbol{c})^{\mathrm{t}}}{|x\pm\boldsymbol{c}|^{2}}\Big),\\ \Delta^{2}\chi_{c}^{\pm}=-\frac{(d-1)(d-3)}{|x\pm\boldsymbol{c}|^{3}},
\hspace{1em} \text{in }B(0,c/4).
\end{gather*}
In particular, it verifies there
\[
D^{2}\chi_{c}\geq0,\ \Delta\chi_{c}\gtrsim\frac{1}{c},\Delta^{2}\chi_{c}\leq0\hspace{1em}\text{in }B(0,c/4),
\]
where the sign of $\Delta^{2}\chi_{c}$ is due to $d\geq3$, and therefore,
by (\ref{eq:with_epsc}), we have, for all $A\leq c/4$
\begin{equation}
z''(t)\gtrsim\frac{1}{c}\int_{B(0,A)}|u|^{\alpha+2}-\int_{\mathbb{R}^{d}}\nabla\chi_{c}\cdot\nabla V|u|^{2}+\frac{1}{c}\epsilon(c).\label{eq:eps_c2}
\end{equation}
Now, as $u\neq0$ and $\left\{ u(t)\text{,\ }t\in\mathbb{R}\right\} $
is relatively compact in $H^{1}$, there exists $\mu>0$ and $A>0$
such that
\[
\sup_{t\in\mathbb{R}}\int_{B(0,A)}|u|^{\alpha+2}\geq2\mu.
\]
We fix such an $A>0$ and take $c>0$ large enough so that $A\leq c/4$
and $|\epsilon(c)|\leq\mu$. Then (\ref{eq:eps_c2}) gives
\begin{equation}
z''(t)\gtrsim\frac{1}{c}\mu-\int_{\mathbb{R}^{d}}\nabla\chi_{c}\cdot\nabla V|u|^{2}.\label{eq:zpp1}
\end{equation}
Let $R$ be a continuous function of $c$ such that $R(c)/c\longrightarrow0$
as $c\longrightarrow+\infty$. In the seek of readability, we will
write $R$ for $R(c)$ in the sequel. Note that, by the Hölder inequality
and because $\nabla\chi_{c}$ is bounded and $|x|^{\beta}\nabla V\in L^{d/2}$
\begin{multline}
\big|\int_{|x|\geq R}\nabla\chi_{c}\cdot\nabla V|u|^{2}\big|\leq\Vert\nabla V\Vert_{L^{d/2}(|x|\geq R)}\Vert u\Vert_{L^{2^{\star}}(|x|\geq R)}^{2}\\
\leq\frac{1}{R^{\beta}}\Vert|x|^{\beta}\nabla V\Vert_{L^{d/2}}\Vert u\Vert_{L^{2^{\star}}(|x|\geq R)}^{2},\label{eq:decrig1}
\end{multline}
but, because $\left\{ u(t)\text{,\ }t\in\mathbb{R}\right\} $ is relatively
compact in $H^{1}$ and by Sobolev embedding in $L^{2^{\star}}$,
\begin{equation}
\sup_{t\in\mathbb{R}}\Vert u\Vert_{L^{2^{\star}}(|x|\geq R)}=\epsilon(R),\label{eq:decrig2}
\end{equation}
where $\epsilon(R)\longrightarrow0$ when $R\longrightarrow+\infty$
and thus, using (\ref{eq:zpp1}), (\ref{eq:decrig1}) and (\ref{eq:decrig2})
\begin{equation}
z''(t)\gtrsim\mu/c-\int_{B(0,R)}\nabla\chi_{c}\cdot\nabla V|u|^{2}+\frac{1}{R^{\beta}}\epsilon(R).\label{eq:zpp2}
\end{equation}
Now, notice that as $R(c)\ll c$, $\chi_{c}$ coincides with (\ref{eq:x_c_naiv})
in $B(0,R)$ and in particular
\[
\nabla\chi_{c}(c)=\frac{x-\boldsymbol{c}}{|x-\boldsymbol{c}|}+\frac{x+\boldsymbol{c}}{|x+\boldsymbol{c}|}\ \text{in }B(0,R).
\]
Because $V_{1}$ and $V_{2}$ are repulsive (assumption (G1)), the
outward-pointing normal to their level surfaces is 
\[
-\frac{\nabla V_{1,2}}{|\nabla V_{1,2}|}.
\]
Thus, by \lemref{poids} applied to the level surfaces of $V_{1}$
and $V_{2}$, together with assumptions (G2) and (G3), we get, in
$B(0,R)$
\[
-\nabla\chi_{c}\cdot\frac{\nabla V_{1,2}}{|\nabla V_{1,2}|}\geq O(\frac{R^{4}}{c^{4}}).
\]
Therefore, by Hölder inequality, Sobolev embedding and conservation
of energy
\begin{multline}
\Big|\int_{|x|\leq R,-\nabla\chi\cdot\nabla V(x)<0}-\nabla\chi_{c}\cdot\nabla V|u|^{2}\Big|\lesssim\frac{R^{4}}{c^{4}}\int\big(|\nabla V_{1}|+|\nabla V_{2}|\big)|u|^{2}\\
\leq\frac{R^{4}}{c^{4}}\big(\Vert\nabla V_{1}\Vert_{L^{d/2}}+\Vert\nabla V_{2}\Vert_{L^{d/2}}\big)\Vert u\Vert_{L^{2^{\star}}}^{2}\\
\leq\frac{R^{4}}{c^{4}}\big(\Vert\nabla V_{1}\Vert_{L^{d/2}}+\Vert\nabla V_{2}\Vert_{L^{d/2}}\big)\Vert u\Vert_{H^{1}}^{2}\lesssim\frac{R^{4}}{c^{4}}E(u_{0})^{2}.\label{eq:rigpartneg}
\end{multline}
Hence, (\ref{eq:zpp2}) together with (\ref{eq:rigpartneg}) gives
\[
z''(t)\gtrsim\frac{\mu}{c}+O(\frac{R^{4}}{c^{4}})+\frac{1}{R^{\beta}}\epsilon(R).
\]
Let us take $R(c)=c^{\nu}$. Then we get
\[
z''(t)\gtrsim\frac{1}{c}(\mu+O(c^{4\nu-3})+c^{1-\beta\nu}\epsilon(c^{\nu})).
\]
Thus, taking
\[
\nu=\frac{1}{\beta},
\]
and assuming 
\[
\nu<\frac{3}{4}\iff\beta>\frac{4}{3},
\]
in such a way that $R(c)/c\longrightarrow0$ and, in particular, $4\nu-3<0$,
we get, for $c>0$ fixed large enough
\[
z''(t)\gtrsim\frac{\mu}{2c},
\]
and (\ref{eq:prig1}) is contradicted.
\end{proof}
Our main result now follows:
\begin{proof}[Proof of \thmref{main}.]
If $E_{c}<\infty$, then \thmref{critsol} allows us to extract a
critical element $\varphi_{c}\in H^{1}$, $\varphi_{c}\neq0$, such
that the corresponding solution $u_{c}$ of (\ref{eq:nlsv}) verifies
that $\left\{ u_{c}(t),\ t\geq0\right\} $ is relatively compact in
$H^{1}$. By \thmref{rigidity}, such a solution cannot exist, so
$E_{c}=\infty$ and by Proposition \ref{pert1}, all the solutions
of (\ref{eq:nlsv}) scatter in $H^{1}$.
\end{proof}

\subsection*{Aknowledgments}

The author is grateful to Fabrice Planchon for many enlightening discussions
about multipliers methods, and to the anonymous referee for their
valuable remarks and suggestions.

\bibliographystyle{amsalpha}
\bibliography{ref}

\end{document}